\documentclass[
	11pt,oneside,english,a4paper]{amsart}
\usepackage[T1]{fontenc}
\usepackage[utf8]{inputenc}
\usepackage[dvipsnames]{xcolor}
\usepackage[yyyymmdd,hhmmss]{datetime}
\usepackage[numbers,sort,compress]{natbib}
\usepackage{amstext,amsthm,amssymb}
\usepackage{xparse,mathrsfs,mathtools}
\usepackage[english]{babel}
\usepackage[pdftex,unicode=true,%
	pdfusetitle,pdfdisplaydoctitle=true,%
	pdfpagemode=UseOutlines,%
	bookmarks=true, bookmarksnumbered=false, bookmarksopen=true, bookmarksopenlevel=2,%
	breaklinks=true,%
	pdfencoding=unicode,psdextra,
	pdfcreator={},
	pdfborder={0 0 0}
]{hyperref}
\usepackage{lmodern,microtype}
\usepackage[a4paper]{geometry}
\usepackage[capitalise,nameinlink]{cleveref}

\hypersetup{
	colorlinks=false,
	pdfborder={0 0 0},
	pdfborderstyle={/S/U/W 0},
}

\newtheorem*{thm*}{Theorem}
\newtheorem{thm}{Theorem}
\newtheorem{lem}[thm]{Lemma}

\newcommand{\NN}{\mathbb{N}}
\newcommand{\ZZ}{\mathbb{Z}}
\newcommand{\dd}[1]{\mathop{\mathrm{d}#1}}
\newcommand{\inv}{\operatorname{inv}}
\DeclarePairedDelimiter\parentheses{\lparen}{\rparen}
\DeclarePairedDelimiter\braces{\lbrace}{\rbrace}
\DeclarePairedDelimiter\abs{\lvert}{\rvert}

\DeclarePairedDelimiter\gen{\langle}{\rangle}
\NewDocumentCommand\set{ s o m o }{%
	\IfBooleanTF{#1}{\IfNoValueTF{#4}{\braces*{#3}}{\braces*{\,#3:#4\,}}}{%
	\IfNoValueTF{#2}{\IfNoValueTF{#4}{\braces{#3}}{\braces{\,#3:#4\,}}}{%
	\IfNoValueTF{#4}{\braces[#2]{#3}}{\braces[#2]{\,#3:#4\,}}}}%
}

\crefname{section}{§}{§§}
\crefname{figure}{Figure}{Figures}
\crefformat{equation}{#2(#1)#3}
\crefformat{enumi}{#2(#1)#3}
\numberwithin{equation}{section}

\title[Farey fraction spin chain]{Remark on the Farey fraction spin chain}
\date{\today{}}


\author{Marc~Technau}
\address{%
	Marc~Technau\\%
	Institut für Analysis und Zahlentheorie\\%
	TU~Graz\\%
	Kopernikusgasse~24/II\\%
	8010~Graz\\%
	Austria}
\email{mtechnau@math.tugraz.at}

\begin{document}
\begin{abstract}
	In \citeyear{KlebanOzluk}, Kleban and Özlük introduced a `Farey fraction spin chain' and made a conjecture regarding its asymptotic number of states with given energy, the latter being given (up to some normalisation) by the number $\Phi(N)$ of $2\times2$ matrices arising as products of $(\!\begin{smallmatrix} 1 & 0 \\ 1 & 1 \end{smallmatrix}\!)$ and $(\!\begin{smallmatrix} 1 & 1 \\ 0 & 1 \end{smallmatrix}\!)$ whose trace equals $N$.
	Although their conjecture was disproved by Peter~(\citeyear{peter2001limit-distribution}), quite precise results are known \emph{on average} by works of Kallies--Özlük--Peter--Snyder~(\citeyear{kallies2001asymptotic-properties}), Boca~(\citeyear{boca2007products-of-matrices}) and Ustinov~(\citeyear{ustinov2013spin-chains}).
	
	We show that the problem of estimating $\Phi(N)$ can be reduced to a problem on divisors of quadratic polynomials which was already solved by Hooley~(\citeyear{hooley1958representation}) in a special case and, quite recently, in full generality by Bykovski{\u{\i}} and Ustinov~(\citeyear{bykovskii2019hooleys-problem}).
	This produces an unconditional estimate for $\Phi(N)$, which hitherto was only (implicitly) known, conditionally on the availability on wide zero-free regions for certain Dirichlet $L$-functions, by the work of Kallies--Özlük--Peter--Snyder.
\end{abstract}
\maketitle

\section{Introduction}

In an effort to introduce further examples pointing towards a connection between the Lee--Yang theory of phase transitions and the Riemann hypothesis, Kleban and Özlük~\cite{KlebanOzluk} have introduced a \emph{`Farey fraction spin chain'}.
Mathematically, their model is concerned with products of the matrices
\[
	A = {\begin{pmatrix} 1 & 0 \\ 1 & 1 \end{pmatrix}}
		\quad\text{and}\quad
	B = {\begin{pmatrix} 1 & 1 \\ 0 & 1 \end{pmatrix}}.
\]
Finite products built from $A$ and $B$ are regarded as \emph{spin states} of that model.
The interested reader is referred to the recent work of Singer~\cite{singer2015positive-limit-fourier-transform} which provides more physical motivation (see, in particular, the references given in §~2 of that paper).
By means of a connection with Farey fractions it was observed in~\cite{KlebanOzluk} that the aforementioned products are unique in the sense that the monoid $\gen{A,B} \subset \mathrm{SL}_2(\ZZ)$ generated by $A$ and $B$ is free over $A$ and $B$ (see also \cite{nathanson2015pairs-of-matrices} for a neat proof of a more general result motivated by the `ping--pong lemma' from geometric group theory).
The \emph{energy} of a spin state $C$ is then given by the logarithm of the trace of the matrix $C$.
It is then natural to ask for the number of spin states with given energy.
This leads one to define
\begin{equation}\label{eq:PhiDef}
	\Phi(N) = \#\set{ C\in \gen{A,B} }[ \operatorname{tr} C = N ],
\end{equation}
which turns out to be finite for $N>2$.
(Mind that this definition differs from the one in~\cite{KlebanOzluk} by a factor of two, but is in accordance with the notation from later works on the topic~\cite{kallies2001asymptotic-properties,peter2001limit-distribution,boca2007products-of-matrices,ustinov2013spin-chains}.)
In~\cite{KlebanOzluk}, the asymptotic formula
\begin{equation}\label{eq:PhiWrongConjecture}
	\Phi(N) \sim c_0 {\mkern 2mu} N \log N, \quad \text{as } N \to \infty,
\end{equation}
with $c_0=1$ was conjectured.
However, Kallies, Özlük, Peter and Synder~\cite{kallies2001asymptotic-properties} later showed that the summatory function $\Psi$ of $\Phi$, given by
\[
	\Psi(N)
	= \sum_{3\leq n \leq N} \Phi(n)
	= \#\set{ C\in \gen{A,B} }[ 3 \leq \operatorname{tr} C \leq N ],
\]
satisfies the asymptotic formula
\begin{equation}\label{eq:Kallies:PsiAsymptotics}
	\Psi(N) = c_1 N^2 \log N + O(N^2 \log \log N),
\end{equation}
with $c_1 = 6/\pi^2 = 1/\zeta(2)$ and $N\geq 3$.
Consequently, if the conjectured asymptotics~\cref{eq:PhiWrongConjecture} were to be correct, then $c_0=2c_1$ instead of $c_0 = 1$.
However, even this modified conjecture was disproved by Peter~\cite{peter2001limit-distribution}, who showed that the normalised quantity $\Phi_*(N) = \Phi(N)/(N\log N)$ admits (with respect to Lebesgue measure) an absolutely continuous, smooth limiting distribution, i.e., there is some function $\delta \in \mathrm{C}^\infty(0,\infty) \cap \mathrm{L}^1(0,\infty)$ such that, for any $0<a<b<\infty$,
\[
	\lim_{N\to\infty} \frac{1}{N} \#\set{ 3 \leq n\leq N }[ \Phi_*(n) \in \lparen a,b\rbrack ]
	= {\int_a^b \delta(t) \dd{t}}.
\]
In particular, $\Phi_*(N)$ does not tend to a limit as $N\to\infty$.
Therefore, indeed, \cref{eq:PhiWrongConjecture} cannot hold for any fixed $c_0$.
\cref{fig:plots} illustrates the rather erratic behaviour of $\Phi_*$.

Subsequently, the focus of investigation has shifted towards the quantity $\Psi(N)$.
Boca~\cite{boca2007products-of-matrices} obtained the two-term asymptotic formula
\begin{equation}\label{eq:PsiAsymptoticsBoca}
	\Psi(N) = c_1 N^2 \log N + c_2 N^2 + O_\epsilon(N^{7/4+\epsilon})
\end{equation}
with\footnote{%
	The value of the constant $c_2$ is misprinted in~\cite{boca2007products-of-matrices}.
	This is corrected in~\cite{ustinov2013spin-chains}.%
}
\[
	c_2 = \frac{1}{\zeta(2)} \parentheses*{ \gamma - \frac{3}{2} - \frac{\zeta'(2)}{\zeta(2)}},
\]
where $\gamma$ denotes the Euler--Mascheroni constant and $\epsilon>0$ is arbitrary.

\begin{figure}%
	\centering%
	\includegraphics{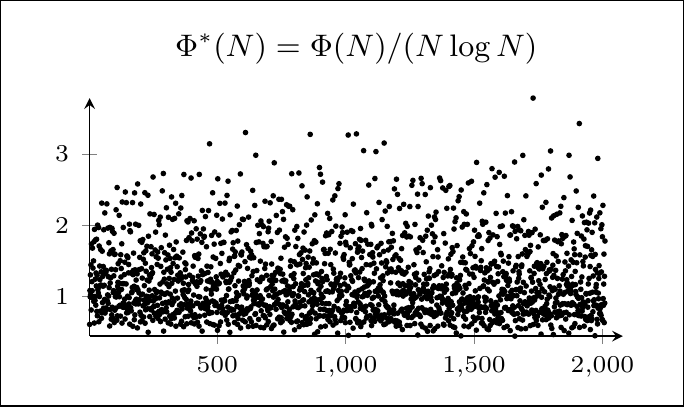}\hspace{3mm}
	\includegraphics{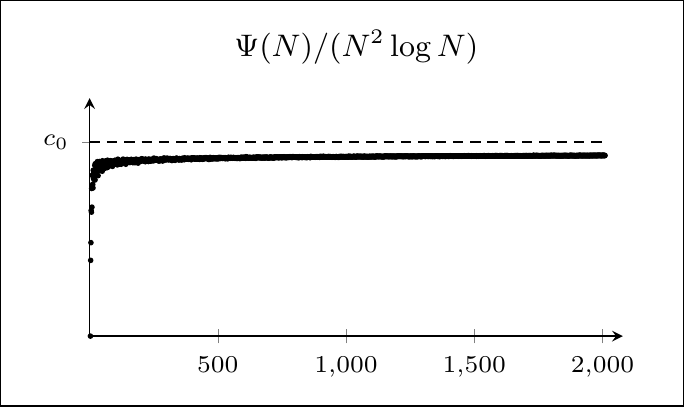}
	\caption{Plots of normalised versions of $\Phi(N)$ and $\Psi(N)$.}%
	\label{fig:plots}%
\end{figure}

The argument in~\cite{kallies2001asymptotic-properties} is quite involved and draws on a connection between matrices in $\gen{A,B}$ and arithmetic in real quadratic fields.
The latter problem could then be addressed by employing $L$-function techniques and a result of Faivre~\cite{faivre1992distribution-levy-const} which was obtained by means of dynamical methods, operator theory and a Tauberian theorem.
On the other hand, Boca's result~\cref{eq:PsiAsymptoticsBoca} was obtained by means of a reducing to a lattice point counting problem, involving modular hyperbolas, whose resolution, in turn, rests upon an application of Weil's bound for Kloosterman sums.
The error term in~\cref{eq:PsiAsymptoticsBoca} was later sharpened by Ustinov~\cite{ustinov2013spin-chains} to $O(N^{3/2}(\log N)^{4})$, essentially, by cleverly injecting a more efficient use of Fourier analysis into the lattice point counting arguments from~\cite{boca2007products-of-matrices} (see also~\cite{ustinov2015points-of}).

\medskip

In the present note, we exploit a formula for $\Phi(N)$ already observed in~\cite[Theorem~4]{KlebanOzluk} (\cref{lem:KlebanOzluk} below) to connect $\Phi(N)$ with classical work of Hooley~\cite{hooley1958representation} concerning the sum
\begin{equation}\label{eq:HooleysProblem}
	\sum_{ 0 \leq \abs{m} < \sqrt{Y} } d(Y - m^2),
\end{equation}
where $d(n)$ denotes the number of positive divisors of $n$.
Hooley's result then immediately yields an `asymptotic formula' for $\Phi(N)$ for \emph{even} $N$.
(The term `asymptotic formula' is in quotes here, because various arithmetic quantities enter the main term.)
Although a careful inspection of Hooley's argument would also furnish such a formula for \emph{odd} $N$ as well, Bykovski{\u{\i}} and Ustinov~\cite{bykovskii2019hooleys-problem} have recently derived a suitable generalisation (and strengthening) of Hooley's result to sums of the shape~\cref{eq:Upsilon} given below.
Their proof uses their result for the reader's convenience.
This requires some notation.
To this end, let $X$ be a positive integer and consider the quantity
\begin{equation}\label{eq:Upsilon}
	\Upsilon(X)
	= \sum_{\substack{ 0 \leq \abs{m} < \sqrt{X} \\ m^2 \equiv X \bmod 4 }} d\parentheses*{\frac{X-m^2}{4}}
	= \#\set{ (\lambda,\mu,m)\in\NN^2\times\ZZ }[ 4\lambda\mu + m^2 = X ].
\end{equation}
Clearly $\Upsilon(X)$ is only non-zero for $X\equiv 0,1\bmod 4$, so we may as well restrict our discussion to those $X$.
(For $X\equiv 0\bmod 4$ the above sum reduces to the sum~\cref{eq:HooleysProblem} studied by Hooley~\cite{hooley1958representation}.)

Recall that any non-zero integer $X\equiv 0,1\bmod 4$ can be written uniquely in the form $X = D r^2$ with some fundamental discriminant $D$ and positive integer $r$.
Consider the Dirichlet $L$-function $L_D$ associated to the Kronecker symbol $\parentheses[\big]{\frac{D}{\cdot}}$.
Moreover, for the above $r$ and $D$, as well as complex $s$, put
\[
	\eta_s(r;D) = \sum_{\ell \mid r} \parentheses*{\frac{D}{\ell}} \frac{\mu(\ell)}{\ell^{2s}} \sum_{d\mid (r/\ell)} d^{1-4s}.
\]
Finally, let $\eta_s'(r;D)$ denote the derivative of $\eta_s(r;D)$ with respect to $s$, and let
\[
	A(X) = \frac{2}{\zeta(2)} X^{1/2} L_D(1) \eta_{1/2}(r;D) \parentheses*{
		\log X
		+ 2 \frac{L_D'(1)}{L_D(1)}
		+ \frac{\eta_{1/2}'(r;D)}{\eta_{1/2}(r;D)}
		+ c_3
	},
\]
where $c_3$ is a certain explicitly given constant, the formula for which we do not reproduce here (see~\cite{bykovskii2019hooleys-problem} for details).
\begin{thm}[Bykovski{\u{\i}} and Ustinov]\label{thm:BykovskiiUstinov}
	There is some positive $\delta$ such that
	\[
		\Upsilon(X) = A(X) + O\parentheses{ X^{1/2-\delta} },
	\]
	as $X\geq 1$ tends to infinity within the residue classes $0, {\mkern 1mu} 1 \bmod 4$.
\end{thm}

\section{Main result}

The main result of the present note shows that $\Phi$ defined in~\cref{eq:PhiDef} is linked to $\Upsilon$ defined in~\cref{eq:Upsilon}.
Precisely, we have the following:

\begin{thm}\label{thm:Main}
	$\Phi(N) = \Upsilon(N^2-4)$ for $N\geq 3$.
\end{thm}

Some precursors of the formula in \cref{thm:Main} can already be found in~\cite{KlebanOzluk} and~\cite{kallies2001asymptotic-properties}, so we try to give the reader a sense of what is new here and what is not.

\medskip

First, in~\cite[Theorem~4]{KlebanOzluk} a connection between $\Phi(N)$ and a restricted divisor sum is proved.
We state this as \cref{lem:KlebanOzluk} below and the authors of~\cite{KlebanOzluk} already point the reader towards the large existing literature on sums with the divisor function.

In~\cite[§~3]{kallies2001asymptotic-properties}, the authors study the arithmetic function $\rho_m\colon\NN\to\NN_0$ given by
\begin{equation}\label{eq:RhoSum}
	\rho_m(a) = \#\set{ b\in\ZZ }[
		0\leq b < 2a, \,
		b^2 \equiv m \bmod 4a, \,
		\gcd(a,b,c)=1, \,
		4ac = b^2 - m
	].
\end{equation}
This function is shown to be multiplicative and, by means of decomposing the associated Dirichlet series into $L$-functions, they prove a formula for
\[
	\sum_{a\leq x} \frac{\rho_\Delta(a)}{a},
\]
where $\Delta>0$ is a non-square (not necessarily fundamental) discriminant and $\Delta/4 \leq x^2 \leq \Delta$, subject to the assumption that a certain Dirichlet $L$-function associated with $\Delta$ admits a sufficiently wide zero-free region (see~\cite[Proposition~3.13]{kallies2001asymptotic-properties}).
The main result of~\cite{kallies2001asymptotic-properties}, the asymptotic formula~\cref{eq:PsiAsymptoticsBoca} stated above, is then obtained by drawing a connection between $\Psi(N)$ and averages of~\cref{eq:RhoSum} and using the aforementioned formula for~\cref{eq:RhoSum} in connexion with zero-density results which, on average, supply the required zero-free regions.

At this point, it is worth pointing out that Hooley~\cite{hooley1958representation} already successfully dealt with sums very similar to~\cref{eq:RhoSum} in~1958 without an appeal to any conjectural zero-free regions for $L$-functions.
This suggested that a successful estimation of $\Phi(N)$ subject to no hitherto unproven hypotheses should be within reach, and \cref{thm:Main} in combination with \cref{thm:BykovskiiUstinov} confirms that this is, indeed, the case.

\medskip

Although for the purpose of analysing $\Psi(N)$, the approach via zero-density methods seems equally capable, it stands to hope that \cref{thm:Main}, by means of combination with \cref{thm:BykovskiiUstinov}, may lead to further insights into the behaviour of $\Phi(N)$ in future works.

\section{Proof of \texorpdfstring{\cref{thm:Main}}{Theorem\autoref{thm:Main}}}

Let $d_n(x) = \#\set{ k \mid x }[ k<n  ]$ count the positive divisors of $x$ that are strictly less than $n$.
Our point of departure shall be the following lemma:
\begin{lem}[Kleban--Özlük]\label{lem:KlebanOzluk}
	\(\displaystyle
		\Phi(N) = 2 \sum_{ 1\leq n < N} d_n( nN - n^2 - 1 )
	\)
	for $N\geq 3$.
\end{lem}
\begin{proof}
	This is \cite[Theorem~4]{KlebanOzluk}, keeping in mind that $\Phi(N)$ in that paper is half of the quantity defined in~\cref{eq:PhiDef}.
	As the style in~\cite{KlebanOzluk} is terse, a reader interested in seeing more details may also consult the very neatly arranged arguments in~\cite{boca2007products-of-matrices}.
	There it is shown that, for $N\geq 3$,
	\begin{equation}\label{eq:PhiSystem}
		\tfrac{1}{2} \Phi(N)
		= \#\set*{ (q,q') : \begin{array}{@{} l @{}}
			1\leq q < q' < N,          \\
			\gcd(q,q') = 1,            \\
			q' + \inv_q(q') = N        
		\end{array} }
		+ \#\set*{ (q,p,t) : \begin{array}{@{} l @{}}
			1 \leq p < q < N,          \,
			t \geq 1,                  \\
			\gcd(p,q) = 1,             \\
			q + \inv_p(q) + pt = N     
		\end{array} },
	\end{equation}
	where, for coprime $x$ and $m$, $\inv_m(x)$ denotes the unique integer $y \in \set{ 1,\ldots,m }$ such that $xy\equiv 1 \bmod m$ (see~\cite[§~2 and the start of §§~3--4 respectively]{boca2007products-of-matrices}; actually, the quantity considered there is $\Psi(N)$, but the necessary changes are only the obvious ones).
	The first term on the right hand side of~\cref{eq:PhiSystem} can be removed if the condition `$t\geq 1$' is replaced by `$t\geq 0$'.
	Suppose that this has been done.
	For $(q,p,t)$ counted by the modified version of the right hand side of~\cref{eq:PhiSystem}, $p$ divides $N - q - \inv_p(q)$.
	In particular, $p$ divides $qN - q^2 - 1$.
	Conversely, for any $1\leq q < N$, every positive divisor $p$ of $qN - q^2 - 1$ is certainly coprime to $q$, and any such divisor with $p<q$ gives rise to some triple $(q,p,t)$ as above by setting $t = (qN - q^2 - \inv_p(q))/p$.
	This proves the lemma.
\end{proof}

Suppose that $N\geq 3$ is some integer.
A linear change in variables in \cref{lem:KlebanOzluk} shows that
\[
	\Phi(N) = 2 \sum_\nu d_{\nu+N/2} \parentheses[\bigg]{\frac{N^2 - 4 - (2\nu)^2}{4}},
\]
where $\nu$ ranges over either integers or $\text{integers}+1/2$ according to whether $N$ is even or odd and such that the numerator in the above fraction is positive.
Consequently, in either case, $m=2\nu$ ranges over the integers such that $m^2 < N^2 - 4$ and $N^2 - 4 \equiv m^2 \bmod 4$.
Hence,
\[
	\Phi(N) = 2 \sum_{\substack{ 0 \leq m^2 < N^2 - 4 \\ m^2 \equiv N^2 \bmod 4 }} d_{(m+N)/2} \parentheses[\bigg]{\frac{N^2 - 4 - m^2}{4}},
\]
We now contend that in the above we may replace $N$ in the subscript of $d$ by $\sqrt{N^2-4}$ without affecting the value of the expression.
Indeed, the error introduced by the above change is precisely twice the number
\[
	\#\set*{ (\lambda,\mu,m) \in\NN^2\times\ZZ }[ {\begin{array}{@{} l @{}}
		m^2 < N^2 - 4,                  \\
		4 \lambda \mu = N^2 - 4 - m^2,  \\
		\sqrt{N^2-4} \leq \lambda-m < N
	\end{array}} ].
\]
However, the last inequality cannot be satisfied for $N\geq 3$, so the above expression is zero.

The above discussion shows that
\begin{equation}\label{eq:PhiRelationToUpsilon}
	\Phi(N) = 2\Upsilon_<(N^2-4),
\end{equation}
where
\begin{equation}\label{eq:UpsilonWithCutOff}
	\Upsilon_<(X) = \sum_{\substack{ 0 \leq \abs{m} < \sqrt{X} \\ m^2 \equiv X \bmod 4 }} d_{(m+\sqrt{X})/2}\parentheses[\bigg]{\frac{X-m^2}{4}}.
\end{equation}
The \emph{proof of \cref{thm:Main}} now follows immediately from~\cref{eq:PhiRelationToUpsilon} and the next lemma:
\begin{lem}\label{lem:KeyLemma}
	$\Upsilon(X) - 2 \Upsilon_<(X) = 0$ if $X$ is not a square and $=\sqrt{X}-1$ otherwise.
\end{lem}
\begin{proof}\newcommand{\cL}{\ensuremath{\mathscr{L}}}%
	Observe that $\Upsilon(X)$ counts precisely the number of elements in the set
	\[
		\cL = \set{ (\lambda,\mu,m) \in \NN^2 \times \ZZ }[ 4\lambda\mu = X-m^2 ].
	\]
	Suppose first that $X$ is \emph{not} a square.
	(In fact, this is the only case we require for the proof of \cref{thm:Main}.)
	Let $\cL_0 = \cL \cap (\NN^2 \times \set{0})$ denote the (possibly empty) set of triples $(\lambda,\mu,m) \in \cL$ with $m=0$.
	As $X$ is not a square, the set $\cL_0^< = \set{ (\lambda,\mu,0) \in \cL_0 }[ 2\lambda < \sqrt{X} ]$ contains exactly half the number of elements of $\cL_0$.
	Next, we partition $\cL_+ = \cL \cap \NN^3$ into three subsets, $\cL_1$, $\cL_2$ and $\cL_3$ (say), according to which of the following intervals $2\lambda$ belongs to:
	\[
		(-\infty, -m + \sqrt{X}), \quad
		[-m + \sqrt{X}, m + \sqrt{X}], \quad
		(m + \sqrt{X}, +\infty).
	\]
	It is easily seen that the map sending $(\lambda,\mu,m)$ to $(\mu,\lambda,m)$ induces a bijection from $\cL_1$ onto $\cL_3$.
	Therefore,
	\begin{equation}\label{eq:UpsilonDoubleCounting}
		\begin{aligned}
			\Upsilon(X) = \#\cL &
			= \#\cL_0 + 2 \, \#\cL_+ \\ &
			= \#\cL_0 + 2 \, \#\cL_1 + 2 \, \#\cL_2 + 2 \, \#\cL_3 \\ &
			= 2 \, \#\cL_0^< + 4 \, \#\cL_1 + 2 \, \#\cL_2.
		\end{aligned}
	\end{equation}
	However, by splitting according to the sign of $m$, we also find that
	\begin{equation}\label{eq:UpsilonCounting}
		\Upsilon_< = \#\cL_0^< + 2 \, \#\cL_1 + \#\cL_2.
	\end{equation}
	This shows $\Upsilon(X) = 2 \Upsilon_<(X)$ if $X$ is not a square, as claimed.
	
	On the other hand, if $X$ \emph{is} a square, then the above argument requires some minor corrections.
	First, $\cL_0$ is either empty (in which case $\#\cL_0 = 2 \, \#\cL_0^<$ for trivial reasons) or the presence of the element $(\sqrt{X/4},\sqrt{X/4},0)$ results in $\#\cL_0 = 2 \, \#\cL_0^< - 1$.
	Similarly, while~\cref{eq:UpsilonDoubleCounting} remains valid, the right hand side of~\cref{eq:UpsilonCounting} also counts the triples
	\[
		\parentheses[\big]{ \tfrac{1}{2}(m+\sqrt{X}), \tfrac{1}{2}(-m+\sqrt{X}), m } \in \cL_2,
	\]
	which ought to be discarded again.
	Taking all of this into account, one easily obtains the assertion of the lemma.
\end{proof}

\section*{ Acknowledgements}

The author takes pleasure in thanking the joint
	FWF--ANR
project
	\emph{ArithRand} ({FWF~I~4945-N} and {ANR-20-CE91-0006})
for financial support,
as well as both referees for their helpful suggestions which lead to an improved exposition.


\vfill%
\end{document}